\documentclass[a4paper,12pt]{article}
\makeatletter
\@addtoreset{footnote}{page}
\makeatother

\usepackage{enumitem}

\usepackage{amsthm}
\usepackage{amsmath,amssymb,latexsym,amsfonts,mathrsfs}

\renewcommand{\Re}{\mathop{\rm Re}}

\newcommand{\eps}{\ensuremath{\varepsilon}}

\renewcommand{\bar}{\overline}

\newcommand{\bC}{\ensuremath{\mathbb{C}}}
\newcommand{\bD}{\ensuremath{\mathbb{D}}}
\newcommand{\bE}{\ensuremath{\mathbb{E}}}

\newcommand{\bN}{\ensuremath{\mathbb{N}}}

\newcommand{\bP}{\ensuremath{\mathbb{P}}}

\newcommand{\bR}{\ensuremath{\mathbb{R}}}

\newcommand{\cB}{\ensuremath{\mathcal{B}}}

\newcommand{\cF}{\ensuremath{\mathcal{F}}}

\newcommand{\cQ}{\ensuremath{\mathcal{Q}}}

\theoremstyle{plain}
\newtheorem{Thm}{Theorem}[section]

\newtheorem{Lem}[Thm]{Lemma}
\newtheorem{Prop}[Thm]{Proposition}
\newtheorem{Cor}[Thm]{Corollary}

\theoremstyle{definition}

\newtheorem{Rem}[Thm]{Remark}

\numberwithin{equation}{section}

\makeatletter
\renewcommand\section{\@startsection {section}{1}{\z@}%
                                   {-3.5ex \@plus -1ex \@minus -.2ex}%
                                   {2.3ex \@plus.2ex}%
                                   {\normalfont\large\bf}}
\makeatother

\makeatletter
\renewcommand\subsection{\@startsection {subsection}{1}{\z@}%
                                   {-3.5ex \@plus -1ex \@minus -.2ex}%
                                   {2.3ex \@plus.2ex}%
                                   {\normalfont\normalsize\bf}}
\makeatother

\begin{document}
\begin{center}
	{\Large \bf
		Existence of quasi-stationary distributions for downward skip-free Markov chains
	}
\end{center}
\begin{center}
	Kosuke Yamato (Kyoto University)
\end{center}
\begin{center}
	{\small \today}
\end{center}

\begin{abstract}
	For downward skip-free continuous-time Markov chains on non-negative integers stopped at zero, existence of a quasi-stationary distribution is studied.
	The scale function for these processes is introduced and the boundary is classified by a certain integrability condition on the scale function, which gives an extension of Feller's classification of the boundary for birth-and-death processes.
	The existence and the set of quasi-stationary distributions are characterized by the scale function and the new classification of the boundary.
\end{abstract}

%%%%% text %%%%%

\section{Introduction}

Let us consider a continuous-time Markov chain $(\{X_{t}\}_{t \geq 0},\{\bP_{x}\}_{x \in \bN}) $ on $\bN := \{ 0,1,2,\cdots \}$ whose downward transitions are \textit{skip-free}, that is,
\begin{align}
	\bP_{z}[ \tau_{y} < \tau_{x}] = 1 \quad \text{for $x, y, z \in \bN$ with $x < y < z$}, \label{downwardSkipfree}
\end{align}
where $\tau_{A} := \inf \{ t > 0 \mid X_{t} \in A \} \ (A \subset \bN)$ denotes the first hitting time of the set $A$ and we write $\tau_{x} := \tau_{\{x\}} \ (x \in \bN)$.
We assume the following conditions:
\begin{align}
	\left\{
	\begin{aligned}
		\text{(i) } &\text{$X$ is irreducible on $\bN \setminus \{0\}$; $\bP_{x}[\tau_{y} < \infty] > 0 \ (x \in \bN \setminus \{0\},\ y \in \bN)$.} \\
		\text{(ii) } &\text{$X$ is regular on $\bN \setminus \{0\}$; $\bP_{x}[\tau_{\bN \setminus \{x\} } > 0] = 1 \ (x \in \bN \setminus \{0\})$.} \\
		\text{(iii) } &\text{The point $0$ is a trap; $ \bP_{0}[X_{t} = 0] = 1$. }
	\end{aligned}
	\right. \label{generalAssumptions}
\end{align}
We say that a probability distribution $\nu$ on $\bN$ is a \textit{quasi-stationary distribution} if it satisfies
\begin{align}
	\nu(y) = \bP_{\nu}[X_{t} = y \mid \tau_{0} > t] \quad \text{for every $y \geq 0$ and $t \geq 0$}, \nonumber
\end{align}
where we denote $\bP_{\nu} := \sum_{y \in \bN} \nu(y) \bP_{y}$.
In the present paper, we give a necessary and sufficient characterization for existence of a quasi-stationary distribution of $X$
and characterize the set of quasi-stationary distributions.
The precise statements will be given in Section \ref{section:mainResults}.

Downward skip-free Markov chains are a class of Markov chains that include several important applied stochastic models, such as birth-and-death processes,
left-continuous random walks and branching processes.
There have been detailed results for the specific class of processes.
For birth-and-death processes, van Doorn \cite{vanDoornQSD-BD} gave a complete characterization of quasi-stationary distributions, i.e., it gave a necessary and sufficient condition for the existence and showed the set of quasi-stationary distributions.
Collet, Mart\'inez and San Mart\'in \cite[Chapter 5]{Quasi-stationary_distributions} also gives a readable presentation on quasi-stationary distributions for birth-and-death processes.
For branching processes, after a pioneering work by Yaglom \cite{yaglomBranching}, Seneta and Vere-Jones \cite{SenetaVereJones} showed there are infinitely many quasi-stationary distributions for discrete-time subcritical branching processes and
Maillard \cite{branchingQSD} gave a complete characterization of quasi-stationary distributions for continuous-time branching processes.
We can refer to \cite[Section 3.3]{branchingQSD} for a history of the studies on quasi-stationary distributions for branching processes.
For general downward skip-free Markov chains, Kijima \cite{KijimaQSD} showed a partial result on existence of a quasi-stationary distribution, which is generalized in the present paper as we will explain in Remark \ref{rem:kijima}.

Our approach is based on the theory of \textit{scale functions}.
Scale functions are extensively used in the studies of birth-and-death processes, one-dimensional diffusions and spectrally one-sided L\'evy processes.
We will introduce a scale function for downward skip-free Markov chains following Noba \cite{NobaGeneralizedScaleFunc} which introduced a scale function for standard Markov processes on an interval without negative jumps.

Let us briefly recall the scale function for spectrally positive L\'evy processes since it provide a guide to our approach.
For details see Bertoin \cite[Chapter VII]{BertoinLevy} and Kyprianou \cite[Chapter 8]{KyprianouText}.
Suppose $Y$ is a spectrally positive L\'evy process, that is, a one-dimensional L\'evy process without negative jumps or monotone paths.
We denote the characteristic exponent by $\psi$:
\begin{align}
	\psi(\beta) := \log \bE_{0}^{Y} [\mathrm{e}^{-\beta Y_{1}}] \quad (\beta \geq 0), \nonumber
\end{align}
where $\bE^{Y}_{x}$ denotes the underlying probability measure of $Y$ starting from $x \in \bR$.
For $q \geq 0$, the $q$-scale function $W_{Y}^{(q)}(x) \ (x \in \bR)$ is given as the unique function which is strictly increasing and continuous on $[0,\infty)$ and satisfies
\begin{align}
	W^{(q)}_{Y}(x) = 0 \quad (x < 0) \quad \text{and} \quad \int_{0}^{\infty}\mathrm{e}^{-\beta x}W^{(q)}_{Y}(x)dx = \frac{1}{\psi(\beta) - q} \quad \text{for large $\beta > 0$}. \nonumber
\end{align}
It is well known that the scale function gives a simple representation for the exit time from an interval and the potential density killed on exiting an interval, that is, the following formulas hold (see e.g., \cite[Theorem 8.1, 8.7]{KyprianouText}): for $a > 0$
\begin{align}
	\bE_{x}[\mathrm{e}^{-q \tau_{0}}, \tau^{Y}_{0} < \tau^{Y}_{[a,\infty)}] = \frac{W_{Y}^{(q)}(a-x)}{W^{(q)}_{Y}(a)} \quad (x \in \bR) \label{exitProbSPL}
\end{align}
and
\begin{align}
	\int_{0}^{\infty} \mathrm{e}^{-qt}\bP_{x}[Y_{t} \in dy, \tau^{Y}_{0} \wedge \tau^{Y}_{[a,\infty)} > t]dt = u^{(q)}(x,y)dy \label{potentialDensitySPL2}
\end{align}
for
\begin{align}
	u^{(q)}(x,y) = \frac{W^{(q)}_{Y}(a-x) W^{(q)}_{Y}(y)}{W^{(q)}_{Y}(a)} - W^{(q)}_{Y}(y-x) \quad (x,y \in [0,a]), \label{potentialDensitySPL}
\end{align}
where $\tau_{A}^{Y} \ (A \subset \bR)$ denotes the first hitting time of the set $A$ for $Y$.

Bertoin \cite{BertoinQSD} has studied existence of a quasi-stationary distribution for spectrally one-sided L\'evy processes killed on exiting a finite interval $[0,a]$.
He showed that the function $q \mapsto W^{(q)}(x)$ can be analytically extended to the entire function for every $x \geq 0$
and proved, under the assumption of the absolute continuity of the transition probability, there exists a unique quasi-stationary distribution $\nu$, which is represented by a scale function (see \cite[Theorem 2]{BertoinQSD}):
\begin{align}
	\nu(dx) = C W^{(-\rho)}_{Y}(x)dx \quad (x \in [0,a]), \quad \bP_{\nu}[Y_{t} \in dx, \tau^{Y}_{0}\wedge\tau^{Y}_{[a,\infty)} > t] = \mathrm{e}^{-\rho t}\nu(dx), \nonumber
\end{align}
where $\rho := \inf \{ \lambda \geq 0 \mid W^{(-\lambda)}_{Y}(a) = 0 \}$ and $C > 0$ is the normalizing constant.
His proof was given by, roughly speaking, analytically extending the potential density formula \eqref{potentialDensitySPL} to $q \in \bC$ and then applying the $R$-theory for general Markov processes by Tuominen and Tweedie \cite{TuominenTweedie}, which gives a general criterion for existence of a unique quasi-stationary distribution.
In Yamato \cite{QSD_SPL}, existence of a quasi-stationary distribution for spectrally positive L\'evy processes killed on exiting the half line $[0,\infty)$ was studied and a necessary and sufficient condition for the existence is obtained.
Though the $R$-theory is not applicable in this case, the scale function also plays a fundamental role and quasi-stationary distributions are represented by scale functions.
In addition, if a quasi-stationary distribution exists, it necessarily follows that there exist infinitely many ones.
They are given by
\begin{align}
	\nu_{\lambda}(dx) := \lambda W^{(-\lambda)}_{Y}(x)dx, \quad \bP_{\nu_{\lambda}}[Y_{t} \in dx, \tau_{0} > t] = \mathrm{e}^{-\lambda t}\nu_{\lambda}(dx) \quad (\lambda \in (0,\lambda_{0}]), \nonumber
\end{align}
where $\lambda_{0} := \sup \{ \lambda \geq 0 \mid W^{(-\lambda)}_{Y}(x) > 0 \quad \text{for every $x > 0$} \}$ (see \cite[Theorem 1.1]{QSD_SPL}).
Noba \cite{NobaGeneralizedScaleFunc} has introduced the \textit{generalized $q$-scale function} $W^{(q)}(x,y) \ (q \geq 0, \ x,y \in I) $ for standard Markov processes on an interval $I$ without negative jumps by the excursion theory and shown that the formula \eqref{exitProbSPL}, \eqref{potentialDensitySPL2} and \eqref{potentialDensitySPL} are generalized to these processes.

Our approach is combining the methods of Bertoin \cite{BertoinQSD} and Noba \cite{NobaGeneralizedScaleFunc}, that is,
we introduce a generalized $q$-scale function $W^{(q)}(x,y) \ (q \geq 0, \ x,y \in \bN)$ for downward skip-free Markov chains in an analogous way to Noba \cite{NobaGeneralizedScaleFunc}.
Then following Bertoin \cite{BertoinQSD}, we show the function $q \mapsto W^{(q)}(x,y)$ can be extended to the entire function. Actually more strongly, we prove $W^{(q)}(x,y)$ is a polynomial of $q$ for fixed $x,y \in \bN$.
Then we investigate existence of a quasi-stationary distribution through the potential density formula corresponding to \eqref{potentialDensitySPL}.
Thanks to the scale function, we can use the similar method as in the case of the spectrally one-sided L\'evy processes in various aspects though in some areas the arguments are not so parallel due to the space-inhomogeneity.
Since existence of a quasi-stationary distribution implies the exponential integrability of $\tau_{0}$,
the recurrence of the process, or almost equivalently, the behavior of the process around $\infty$ is important.
To see this point, we classify the boundary $\infty$ as \textit{entrance} or \textit{non-entrance} through an integrability on the $0$-scale function, which is an extension of Feller's classification of the boundary for birth-and-death process as we will mention in Remark \ref{rem:classificationIsExtension}.

\subsection{Main results} \label{section:mainResults}

To state our main results, we prepare some notation.
Fix a continuous-time Markov chain $(\{ X_{t} \}_{t \geq 0},\{\bP_{x}\}_{x \in  \bN})$ on $\bN$
whose $Q$-matrix $Q := (Q(x,y))_{x,y \geq 0}$ is conservative, that is, $\sum_{y \in \bN} Q(x,y) = 0 \ (x \in \bN)$, and which is killed at the explosion time $\tau_{\infty} := \lim_{z \to \infty}\tau_{z}^{+}$, where $\tau_{z}^{+} := \tau_{[z,\infty) \cap \bN}$.
Set $Q(x) := - Q(x,x) = \sum_{y \neq x} Q(x,y) \ (x \in \bN)$.
We suppose $X$ satisfies \eqref{downwardSkipfree} and \eqref{generalAssumptions}.
Note that the downward skip-free property \eqref{downwardSkipfree} is characterized by
\begin{align}
	Q(x,y) = 0 \quad \text{for $0 \leq y \leq x-2$} \nonumber
\end{align}
and the property that $0$ is a trap is characterized by
\begin{align}
	Q(0,x) = 0 \quad (x \in \bN). \nonumber 	
\end{align}
Also note that from the conditions (i) and (ii) in \eqref{generalAssumptions}, it holds $Q(x) \in (0,\infty)$ for every $x \geq 1$.
For $x \geq 0$, define the \textit{local time} at $x$ by
\begin{align}
	L^{x}_{t} := \int_{0}^{t}1\{X_{s} = x\}ds. \nonumber
\end{align}
Then we have the occupation time formula and the potential density formula:
for every non-negative function $f: \bN \to [0,\infty]$ it holds
\begin{align}
	\int_{0}^{t}f(X_{s})ds = \sum_{x \geq 0} f(x)L^{x}_{t} \nonumber
\end{align}
and
\begin{align}
	\int_{0}^{\infty}\mathrm{e}^{-qt}\bE_{x}[f(X_{t})]dt  = \sum_{y \geq 0} f(y) \bE_{x}\left[ \int_{0}^{\infty}\mathrm{e}^{-qt}dL^{y}_{t} \right]. \nonumber
\end{align}
Let $\bD$ denote the set of c\`adl\`ag paths on $\bN$.
From now on, we assume the process $X$ is a canonical process on $(\bN^{[0,\infty)},\cB(\bN^{[0,\infty)}) )$ to avoid unnecessary technicalities.
We define the excursion measure $n_{x}(de)$ of $X$ away from $x > 0$ by
\begin{align}
	n_{x}[X \in de] := Q(x)\bP_{x}[ (X \circ \theta_{\tau_{\bN \setminus \{x\} }})_{\cdot \wedge \tau_{x}} \in de] \quad (e \in \bD), \nonumber
\end{align}
where $\theta$ denotes the shift operator.
Under this excursion measure, it is not difficult to check that the \textit{inverse local time} $\eta^{x}$ of $x$, the right-continuous inverse of $t \mapsto L^{x}_{t}$, satisfies
\begin{align}
	-\log \bE_{x}[\mathrm{e}^{-q\eta^{x}(1)}] = q + n_{x}[1 - \mathrm{e}^{-q \tau_{x}}]. \nonumber
\end{align}
Following Noba \cite{NobaGeneralizedScaleFunc}, we introduce the scale function.
For $q \geq 0$, define a function $W^{(q)}: \bN \times \bN \to [0,\infty)$ by
\begin{align}
	W^{(q)}(x,z) := \frac{1}{n_{z}[\mathrm{e}^{-q\tau_{x}}, \tau_{x} < \infty]} \quad \text{for } x < z
	\label{scaleFunc}
\end{align}
and
\begin{align}
	W^{(q)}(x,z) := 0 \quad \text{for } x \geq z. \label{scaleFunc-offDiagonal}
\end{align}
We especially write $W := W^{(0)}$.
We call $W^{(q)}$ the \textit{$q$-scale function}.
We will derive some basic formulas on the scale function in Section \ref{section:scaleFunc}.
Our first main result is to show the scale function $W^{(q)}(x,y) \ (x,y \geq 0, \ q \geq 0)$ is a polynomial of $q$, which enables us to analytically extend the domain of $q$ to $\bC$.
We say that a $\bN \times \bN$-matrix $M = (M(x,y))_{x,y \in \bN}$ is \textit{triangular} if $M(x,y) = 0$ for $x \geq y$.
The proof of Theorem \ref{thm:polynomialRepresentation} will be given in Section \ref{section:polynomialRepresentation}.

\begin{Thm} \label{thm:polynomialRepresentation}
	Let $W^{(q)} = (W^{(q)}(x,y))_{x,y \geq 0}$.
	For $q \geq 0$, the matrix $F = W^{(q)}$ is the unique triangular matrix solution of the following equation:
	\begin{align}
		Q F(x,y) &= I(x,y) + q F(x,y) \quad ( x \geq 1, y \geq 0), \label{eigenFunc}
	\end{align}
	where $I := (\delta_{xy})_{x,y \geq 0}$.
	It has a polynomial representation:
	\begin{align}
		W^{(q)}(x,y)  = \sum_{n \geq 0}q^{n} W^{n+1} (x,y) = \sum_{0 \leq n \leq y-x-1}q^{n} W^{n+1}(x,y) \quad (x,y \geq 0) \label{polynomialRepresentation}
	\end{align}
	(hence $q \mapsto W^{(q)}(x,y)$ is an entire function), where $W^{n}$ denotes the $n$-th product of the matrix $W = W^{(0)}$.
	In addition, for $q \in \bC$, the matrix $I - qW$ is invertible and it holds
	\begin{align}
		W^{(q)} = (I - qW)^{-1}W = W (I - qW)^{-1}. \label{q-scaleRepresentation}
	\end{align}
\end{Thm}

We go on to state our results on existence of a quasi-stationary distribution.
Define the \textit{decay parameter}
\begin{align}
	\lambda_{0} := \sup \{ \lambda \geq 0 \mid \bE_{x}[\mathrm{e}^{\lambda \tau_{0}}] < \infty \quad \text{for some $x \geq 1$} \}. \label{lambdaZero}
\end{align}
We remark that from \cite[Theorem 1]{KingmanRrecurrence}, it always holds $\lambda_{0} < \infty$
and $\bE_{x}[\mathrm{e}^{(\lambda_{0} - \eps)\tau_{0}}] < \infty$ for every $x \geq 1$ and $\eps > 0$.
We also remark that if $\nu$ is a quasi-stationary distribution, the distribution $\bP_{\nu}[\tau_{0} \in dt]$ is exponentially distributed, which can be easily seen from the Markov property.
Thus, the positivity of $\lambda_{0}$ is a necessary condition for existence of a quasi-stationary distribution. 
We introduce a \textit{classification of the boundary} $\infty$ by the integrability of the function $W(0,\cdot)$.
We say that the boundary $\infty$ is \textit{entrance} if it holds
\begin{align}
	\sum_{y \geq 0} W(0,y) < \infty, \label{entrance}
\end{align}
and we say the boundary $\infty$ is \textit{non-entrance} if it holds
\begin{align}
	\sum_{y \geq 0} W(0,y) = \infty. \label{natural}
\end{align}
The following theorem characterizes existence of a quasi-stationary distribution and determines the set of quasi-stationary distributions.
The positivity and integrability of the function $W^{(-\lambda)}(0,\cdot) \ (\lambda \in (0,\lambda_{0}])$ will be shown in \eqref{eq54} and Corollary \ref{cor:ScaleDefinesFiniteMeasure}.
The proof of Theorem \ref{thm:QSDExistenceCharacterization} will be given in Section \ref{section:QSD}.

\begin{Thm} \label{thm:QSDExistenceCharacterization}
	Assume $X$ certainly hits $0$:
	\begin{align}
		\bP_{x}[\tau_{0} < \infty] = 1 \quad (x \geq 0). \label{nonExplosive}
	\end{align}
	Let $\cQ$ denote the set of quasi-stationary distributions.
	Then the following holds depending on the classification of the boundary $\infty$:
	\begin{enumerate}
		\item If the boundary $\infty$ is entrance, it holds $\lambda_{0} > 0$ and $\cQ = \{ \nu_{\lambda_{0}} \}$.
		\item If the boundary $\infty$ is non-entrance and $\lambda_{0} > 0$, it holds $\cQ = \{ \nu_{\lambda} \}_{\lambda \in (0,\lambda_{0}]}$.
	\end{enumerate}
	Here for $\lambda \in (0,\lambda_{0}]$, the probability distribution $\nu_{\lambda}$ is
	\begin{align}
		\nu_{\lambda}(x) := \lambda W^{(-\lambda)}(0,x) \quad (x \geq 0). \nonumber
	\end{align}
\end{Thm}

\begin{Rem} \label{rem:classificationIsExtension}
	Our new classification of the boundary defined in \eqref{entrance} and \eqref{natural} is an extension of Feller's for birth-and-death processes (see e.g., \cite[Chapter 8]{Anderson}).
	Let us consider a birth-and-death process whose $Q$-matrix $Q = (Q(x,y))_{x,y \in \bN}$ is given by
	\begin{align}
		Q(x,y) = \left\{
			\begin{aligned}
				&\mu(x) & (y = x-1), \\
				&-(\mu(x) + \lambda(x)) & (y = x), \\
				&\lambda(x) & (y = x+1), \\
				&0 & (\text{otherwise})
			\end{aligned}
		\right. \nonumber
	\end{align} 
	for $\lambda(x), \mu(x) > 0 \ (x > 0)$ and $\mu(0) = \lambda(0) = 0$.
	Define the \textit{speed measure} $\pi = (\pi(x))_{x \geq 1}$ by
	\begin{align}
		\pi(1) := 1, \quad \pi(x) := \frac{\lambda(1)\lambda(2) \cdots \lambda(x-1)}{\mu(2)\mu(3) \cdots \mu(x)} \quad  (x \geq 2) \nonumber
	\end{align}
	and the (usual) scale function $s: \bN \to [0,\infty)$ by
	\begin{align}
		s(0) = 0, \quad s(x) = \frac{1}{\mu(1)} + \sum_{1 \leq y \leq x-1} \frac{1}{\pi(y) \lambda(y)} \quad (x \geq 1). \nonumber
	\end{align}
	Then the $0$-scale function $W = W^{(0)}$ satisfies
	\begin{align}
		W(0,x) = s(x)\pi(x) \quad (x \geq 0). \nonumber
	\end{align}
	For this process, the entrance condition \eqref{entrance} is equivalent to
	\begin{align}
		\sum_{x > 0}s(x)\pi(x) < \infty. \label{eq20}
	\end{align}
	Note that the boundary $\infty$ is entrance or non-entrance in the sense of Feller's classification \cite[p.262]{Anderson} according as the LHS of \eqref{eq20} is finite or infinite.
\end{Rem}

\begin{Rem} \label{rem:kijima}
	We mention a previous study by Kijima \cite{KijimaQSD} for existence of a quasi-stationary distribution for downward skip-free Markov chains.
	His result \cite[Theorem 3.3]{KijimaQSD} shows existence of infinitely many quasi-stationary distributions under the assumptions $\lambda_{0} > 0$, the process being uniformizable, i.e., $\sup_{x > 0}Q(x) < \infty$, and some technical conditions.
	Theorem \ref{thm:QSDExistenceCharacterization} generalizes his result since the uniformizability implies the non-entrance condition.
	Indeed, if the process is uniformizable, from the definition of the scale function, it holds for $x > 0$
	\begin{align}
		W(0,x) &= \frac{1}{n_{x}[\tau_{0}  <\infty]} \geq \frac{1}{ \sup_{y > 0}Q(y)} > 0, \nonumber
	\end{align}
	and it obviously implies $\sum_{x > 0}W(0,x) = \infty$.
\end{Rem}

\begin{Rem}
	From the computation in Remark \ref{rem:kijima}, we see
	\begin{align}
		\sup_{y > 0} \frac{Q(y)}{y} < \infty \nonumber
	\end{align}
	is a sufficient condition for the non-entrance condition.
	The condition holds, for example, when $X$ is a random walk ($Q(x) = c \ (x \geq 1)$ for some $c > 0$) or a branching process ($Q(x) = cx \ (x \geq 1)$ for some $c > 0$). 
\end{Rem}

From Theorem \ref{thm:QSDExistenceCharacterization}, we see the following.

\begin{Cor}
	Suppose \eqref{nonExplosive} holds.
	Then a quasi-stationary distribution exists if and only if $\lambda_{0} > 0$.
\end{Cor}

Under the entrance condition, we may further show the unique quasi-stationary distribution is the \textit{Yaglom limit}.

\begin{Thm} \label{thm:YaglomSimplified}
	Suppose \eqref{nonExplosive} holds and the boundary $\infty$ is entrance.
	Then the quasi-stationary distribution $\nu_{\lambda_{0}}$ is the Yaglom limit:
	\begin{align}
		\lim_{t \to \infty} \bP_{x}[X_{t} = y \mid \tau_{0} > t] = \nu_{\lambda_{0}}(y) \quad (x,y > 0). \nonumber
	\end{align}
\end{Thm}

Theorem \ref{thm:YaglomSimplified} will be proven in Section \ref{section:YaglomLimit} in more detailed form as Theorem \ref{thm:YaglomLimit}.

The set of quasi-stationary distributions is totally ordered by a stochastic order.
For probability distributions $\mu,\nu$ on $\bN$,
 we say $\mu$ is smaller than $\nu$ in the \textit{likelihood ratio order} and denote by $\mu \leq_{\mathrm{lr}} \nu$ when
\begin{align}
	\mu(x)\nu(x+1) \geq \mu(x+1)\nu(x) \quad \text{for every $x \geq 0$}. \nonumber
\end{align}
It is worth noting that the order $\leq_{\mathrm{lr}}$ is stronger than the usual stochastic order, that is, $\mu \leq_{\mathrm{lr}} \nu$ implies $\sum_{y \geq x} \mu(y) \leq \sum_{y \geq x}\nu(y)$ for every $x \geq 0$ (see e.g., \cite[Theorem 1.C.1]{StochasticOrder}).
We remark that the similar result with the following theorem was given in \cite[Theorem 3.4]{KijimaQSD}.

\begin{Thm} \label{thm:stochasticOrdering}
	Suppose \eqref{nonExplosive} holds, $\lambda_{0} > 0$ and the boundary $\infty$ is non-entrance.
	Then it holds
	\begin{align}
		\nu_{\lambda} \leq_{\mathrm{lr}} \nu_{\lambda'} \quad (0 < \lambda' \leq \lambda \leq \lambda_{0}). \nonumber
	\end{align}
	In particular, the set of quasi-stationary distributions is totally ordered by the likelihood ratio order and the distribution $\nu_{\lambda_{0}}$ is the minimum element.
\end{Thm}

We may characterize the value $\lambda_{0}$ by a positivity of the scale function.

\begin{Thm} \label{thm:spectralBottomCharacterization}
	Suppose \eqref{nonExplosive} holds.
	Then it holds
	\begin{align}
		\lambda_{0} = \max \{ \lambda \geq 0 \mid W^{(-\lambda)}(0,x) > 0 \ \text{for every $x > 0$} \}. \label{eq42}
	\end{align}
\end{Thm}

The proof of Theorems \ref{thm:stochasticOrdering} and \ref{thm:spectralBottomCharacterization} will be given in Section \ref{section:QSD}.

\subsection*{Outline of the paper}

In Section \ref{section:scaleFunc}, we will derive some basic properties of the scale function.
In Section \ref{section:polynomialRepresentation}, we will prove Theorem \ref{thm:polynomialRepresentation}.
In Section \ref{section:QSD}, we will study existence of a quasi-stationary distribution.
In Section \ref{section:YaglomLimit}, we will consider the Yaglom limit under the entrance condition.
In Appendix, we will see that under the entrance condition downward skip-free Markov chains are extended to Feller processes on $\bN \cup \{\infty\}$.

\subsection*{Acknowledgements}

The author would like to thank Kouji Yano who read an early draft of the present paper and gave him valuable comments. This work was supported by JSPS KAKENHI Grant
Number JP21J11000, JSPS Open Partnership Joint Research Projects Grant Number
JPJSBP120209921 and the Research Institute for Mathematical Sciences, an International Joint Usage/Research Center located in Kyoto University.

\section{Scale function and potential density} \label{section:scaleFunc}

Following Noba \cite{NobaGeneralizedScaleFunc}, we show the exit time from an interval and the potential density killed on exiting an interval are represented by the scale function.
Although most of the results in this section are given in \cite{NobaGeneralizedScaleFunc} for standard Markov processes on an interval without negative jumps, and by almost the same argument we can show the corresponding results for downward skip-free Markov chains, we prove some of them for completeness.

\begin{Prop} \label{prop:exitProblem1}
	Let $q \geq 0$.
	For $x,y,z \in \bN$ with $x < y < z $,
	it holds 
	\begin{align}
		\bE_{y}[\mathrm{e}^{-q \tau_{x}}, \tau_{x} < \tau_{z}^{+}]
		= \frac{ n_{z}[\mathrm{e}^{-q \tau_{x}}, \tau_{x} < \infty ] }{ n_{z}[\mathrm{e}^{-q \tau_{y}}, \tau_{y} < \infty ] } = \frac{W^{(q)}(y,z)}{W^{(q)}(x,z)}. \nonumber
	\end{align}
\end{Prop}

\begin{proof}
	By the downward skip-free property and the strong Markov property, it holds for $0 \leq x < y < z$
	\begin{align}
		n_{z}[\mathrm{e}^{-q \tau_{x}}, \tau_{x} < \infty] 
		= n_{z}[\mathrm{e}^{-q \tau_{y}}, \tau_{y} < \infty ]\cdot \bE_{y}[\mathrm{e}^{-q \tau_{x}}, \tau_{x} < \tau_{z}^{+} ] \nonumber.
	\end{align}
\end{proof}

From Proposition \ref{prop:exitProblem1}, we see the scale function derives a martingale.

\begin{Prop}
	For $q \geq 0$ and $0 \leq x < y < z$, the process
	\begin{align}
		Z_{t} := \mathrm{e}^{-qt \wedge \tau_{x} \wedge \tau_{z}^{+}} W^{(q)}(X_{t \wedge \tau_{x} \wedge \tau_{z}^{+}},z) \quad (t \geq 0) \nonumber
	\end{align}
	is a martingale under $\bP_{y}$.
\end{Prop}

\begin{proof}	
	From Proposition \ref{prop:exitProblem1} and the Markov property it holds for $\tau := \tau_{x}\wedge \tau_{z}^{+}$
	\begin{align}
		\bE_{y}[Z_{t} ] &= W^{(q)}(x,z) \bE_{y}[ \mathrm{e}^{-q(t \wedge \tau)} \bE_{{X_{t \wedge \tau}}}[\mathrm{e}^{-q\tau_{x}}, \tau_{x} < \tau_{z}^{+}]] \nonumber  \\
		&= W^{(q)}(x,z) (\bE_{y}[\mathrm{e}^{-q\tau_{x}},\tau_{x} < \tau_{z}^{+}, \tau \leq t] + \bE_{y}[\mathrm{e}^{-q\tau_{x}}, \tau_{x} < \tau_{z}^{+}, \tau > t]) \nonumber \\
		&= W^{(q)}(y,z), \nonumber
	\end{align}
	and similarly we see for $\cF_{u} := \sigma (X_{v},0 \leq v \leq u) \ (u \geq 0)$ and $0 < s < t$
	\begin{align}
		\bE_{y}[Z_{t} \mid \cF_{s}] &= W^{(q)}(x,z) \bE_{y}[\mathrm{e}^{-q(t \wedge \tau)} \bE_{X_{t \wedge \tau}}[\mathrm{e}^{-q\tau_{x}}, \tau_{x} < \tau_{z}^{+}]  \mid \cF_{s}] \nonumber \\
		&= \mathrm{e}^{-q(t \wedge \tau)}W^{(q)}(x,z)  1\{ \tau \leq s \} + \mathrm{e}^{-qs} \bE_{X_{s}}[ Z_{t-s}] 1\{\tau > s\} \nonumber \\
		&= Z_{s} 1\{ \tau \leq s \} + \mathrm{e}^{-qs} W^{(q)}(X_{s},z) 1\{\tau > s\} \nonumber \\
		&= Z_{s}. \nonumber
	\end{align}
\end{proof}

By the same argument as \cite[Lemma 3.5]{NobaGeneralizedScaleFunc},
we see the following.

\begin{Prop} \label{prop:exitProblem2}
	Let $q \geq 0$.
	For $0 \leq x < y < z$ it holds
	\begin{align}
		\bE_{y}[\mathrm{e}^{-q \tau_{x}}, \tau_{x} < \tau_{z}^{+}] = n_{y}[\mathrm{e}^{-q\tau_{x}}, \tau_{x} < \infty ] \cdot \bE_{y}\left[ \int_{0}^{\tau_{x} \wedge \tau_{z}^{+}} \mathrm{e}^{-qt} dL^{y}_{t} \right]. \nonumber
	\end{align}
\end{Prop}

% \begin{proof}
% 	Let $\tau_{x}^{(n)} \ (n \geq 1)$ be the $n$-th return time of $x$ and set $\tau^{(0)}_{x} := 0$.
% 	From the strong Markov property, we have
% 	\begin{align}
% 		\bE_{x}[\mathrm{e}^{-q \tau_{o}}, \tau_{o} < \tau_{A}^{+}]
% 		&= \sum_{n \geq 0}\bE_{x}[\mathrm{e}^{-q \tau_{o}}, \tau_{x}^{(n)} < \tau_{o} < \tau^{(n+1)}_{x}, \tau_{o} < \tau_{A}^{+}] \nonumber \\
% 		&= \sum_{n \geq 0} (\bE_{x}[\mathrm{e}^{-q\tau_{x}^{(1)}},\tau_{x}^{(1)} < \tau_{o} \wedge \tau_{A}^{+}])^{n}
% 		\cdot\bE_{x}[\mathrm{e}^{-q \tau_{o}}, \tau_{o} < \tau_{x}^{(1)} \wedge \tau_{A}^{+}]. \label{eq01}
% 	\end{align}
% 	On the other hand, again from the strong Markov property, it holds
% 	\begin{align}
% 		\bE_{x}\left[ \int_{0-}^{\tau_{o} \wedge \tau_{A}^{+}} \mathrm{e}^{-qt} dL^{x}_{t} \right]
% 		&= \delta_{x} \sum_{n \geq 0} \bE_{x}\left[ \mathrm{e}^{-q \tau^{(n)}_{x}}, \tau^{(n)}_{x} < \tau_{o} \wedge \tau_{A}^{+}  \right] \nonumber \\
% 		&= \delta_{x} \sum_{n \geq 0} (\bE_{x}\left[ \mathrm{e}^{-q \tau_{x}^{(1)}}, \tau_{x}^{(1)} < \tau_{o} \wedge \tau_{A}^{+} \right])^{n}, \label{eq02}
% 	\end{align}
% 	where $\delta_{x}$ is the stagnancy rate of $x$.
% 	Since it holds $\bE_{x}[\mathrm{e}^{-q\tau_{o}}, \tau_{o} < \tau^{(1)}_{x} \wedge \tau_{A}^{+}] = \ell_{x} n_{x}[\mathrm{e}^{-q\tau_{o}}, \tau_{o} < \tau_{A}^{+} ]$, we obtain the desired result from \eqref{eq01} and \eqref{eq02}.
% \end{proof}

Combining the above results, we obtain the representation of the occupation density on an interval by the scale function.
\begin{Thm} \label{thm:potentialDensity}
	Let $q \geq 0$.
	For $0 \leq x < u,y < z$, it holds
	\begin{align}
		\bE_{y}\left[ \int_{0}^{\tau_{x} \wedge \tau_{z}^{+}} \mathrm{e}^{-qt} dL^{u}_{t} \right]
		= &\frac{W^{(q)}(x,u)W^{(q)}(y,z)}{W^{(q)}(x,z)} - W^{(q)}(y,u) \nonumber \\
		= &W^{(q)}(x,u) \bE_{y}[\mathrm{e}^{-q\tau_{x}}, \tau_{x} < \tau_{z}^{+}] - W^{(q)}(y,u). \label{potentialDensity}
	\end{align}
\end{Thm}

\begin{proof}
	The case $u = y$ follows from Propositions \ref{prop:exitProblem1} and \ref{prop:exitProblem2}.
	When $u \neq y$, it holds
	\begin{align}
		\bE_{y}\left[ \int_{0}^{\tau_{x} \wedge \tau_{z}^{+}} \mathrm{e}^{-qt} dL^{u}_{t} \right] &=
		\bE_{y}\left[ \int_{0}^{\tau_{x} \wedge \tau_{z}^{+}} \mathrm{e}^{-qt} dL^{u}_{t}, \tau_{u} < \tau_{x} \wedge \tau_{z}^{+} \right] \nonumber \\
		&= \bE_{y}[\mathrm{e}^{-q\tau_{u}}, \tau_{u} < \tau_{x} \wedge \tau_{z}^{+}]\cdot \bE_{u}\left[ \int_{0}^{\tau_{x} \wedge \tau_{z}^{+}}\mathrm{e}^{-qt}dL^{u}_{t} \right] \nonumber \\
		&= \bE_{y}[\mathrm{e}^{-q\tau_{u}}, \tau_{u} < \tau_{x} \wedge \tau_{z}^{+}] \cdot \frac{W^{(q)}(x,u)W^{(q)}(u,z)}{W^{(q)}(x,z)}. \label{eq06}
	\end{align}
	Thus, we focus on the first term of \eqref{eq06}.
	If $y > u$, it follows from Proposition \ref{prop:exitProblem1} that
	\begin{align}
		\bE_{y}[\mathrm{e}^{-q\tau_{u}}, \tau_{u} < \tau_{x} \wedge \tau_{z}^{+}] 
		= \bE_{y}[\mathrm{e}^{-q\tau_{u}}, \tau_{u} < \tau_{z}^{+}] = \frac{W^{(q)}(y,z)}{W^{(q)}(u,z)}. \nonumber
	\end{align}
	If $y < u$, it follows from the downward skip-free property
	\begin{align}
		\bE_{y}[\mathrm{e}^{-q\tau_{u}}, \tau_{u} < \tau_{x} \wedge \tau_{z}^{+}] \nonumber
		= &\frac{\bE_{y}[\mathrm{e}^{-q\tau_{x}}, \tau_{u} < \tau_{x} < \tau_{z}^{+}]}{\bE_{u}[\mathrm{e}^{-q\tau_{x}}, \tau_{x} < \tau_{z}^{+}]} \nonumber \\
		= &\frac{\bE_{y}[\mathrm{e}^{-q\tau_{x}}, \tau_{u}^{+} < \tau_{x} < \tau_{z}^{+}]}{\bE_{u}[\mathrm{e}^{-q\tau_{x}}, \tau_{x} < \tau_{z}^{+}]} \nonumber \\
		= &\frac{\bE_{y}[\mathrm{e}^{-q\tau_{x}}, \tau_{x} < \tau_{z}^{+}] - \bE_{y}[\mathrm{e}^{-q\tau_{x}}, \tau_{x} < \tau_{u}^{+}]}{\bE_{u}[\mathrm{e}^{-q\tau_{x}}, \tau_{x} < \tau_{z}^{+}]} \nonumber \\
		= &\frac{ W^{(q)}(y,z) / W^{(q)}(x,z) - W^{(q)}(y,u) / W^{(q)}(x,u) }{W^{(q)}(u,z) / W^{(q)}(x,z) }. \nonumber
	\end{align}
\end{proof}

We have the following corollary by the same argument with \cite[Corollary 3.7]{NobaGeneralizedScaleFunc}.

\begin{Cor}
	For $q \geq 0$, define
	\begin{align}
		Z^{(q)}(x,y) :=
		\left\{
		\begin{aligned}
			&1 + q \sum_{x < u < y}W^{(q)}(x,u) & (x < y) \\
			&0 & (x \geq y)
		\end{aligned}
		\right.
		.
		\label{Zfunc}
	\end{align}
	Then it holds for $x < y < z$
	\begin{align}
		\bE_{y}[\mathrm{e}^{-q\tau_{z}^{+}}, \tau_{z}^{+} < \tau_{x}] = Z^{(q)}(y,z) - \frac{W^{(q)}(y,z)}{W^{(q)}(x,z)} Z^{(q)}(x,z). \label{eq51}
	\end{align}
\end{Cor}

\begin{Rem}
	From Theorem \ref{thm:polynomialRepresentation}, the function $q \mapsto Z^{(q)}(x,y)$ is actually a polynomial of $q$ for every fixed $x,y \geq 0$. 
\end{Rem}

\section{Proof of Theorem \ref{thm:polynomialRepresentation}} \label{section:polynomialRepresentation}

We prove Theorem \ref{thm:polynomialRepresentation}.

\begin{proof}[Proof of Theorem \ref{thm:polynomialRepresentation}]
	Since it holds $\tau_{y-1} 1\{\tau_{y-1} < \infty \} = 0 \ \  n_{y}$-a.e. from the downward skip-free property, we see
	\begin{align}
		W^{(q)}(y-1,y) = \frac{1}{n_{y}[\tau_{y-1} = 0]} = \frac{Q(y)}{Q(y)Q(y,y-1)} = \frac{1}{Q(y,y-1)} \quad (y \geq 1). \nonumber
	\end{align}
	Thus, we have
	\begin{align}
		QW^{(q)}(y,y) = Q(y,y-1)W^{(q)}(y-1,y) = 1 \quad (y \geq 1). \label{eq12}
	\end{align}
	Take $y \geq 1$.
	Let $k \geq 1$.
	From Proposition \ref{prop:exitProblem1} it holds
	\begin{align}
		W^{(q)}(y-1,y+k) = \frac{W^{(q)}(y,y+k)}{\bE_{y}[\mathrm{e}^{-q\tau_{y-1}}, \tau_{y-1} < \tau_{y+k}] }. \label{eq07}
	\end{align}
	By considering the one-step transition from $y$, we see 
	\begin{align}
		&\bE_{y}[\mathrm{e}^{-q\tau_{y-1}}, \tau_{y-1} < \tau_{y+k}]  \nonumber \\
		= &\frac{1}{q + Q(y)}\left( Q(y,y-1) + \sum_{1 \leq j \leq k-1} Q(y,y+j) \bE_{y+j}[\mathrm{e}^{-q\tau_{y-1}},\tau_{y-1} < \tau_{y+k}] \right) \nonumber \\
		= &\frac{1}{q + Q(y)}\left( Q(y,y-1) + \sum_{1 \leq j \leq k-1} Q(y,y+j) \frac{W^{(q)}(y+j,y+k)}{W^{(q)}(y-1,y+k)} \right). \label{eq08}
	\end{align}
	Thus, from \eqref{eq07} and \eqref{eq08}, we have
	\begin{align}
		\begin{split}
			0 = &W^{(q)}(y-1,y+k)
			\cdot \left(\frac{Q(y,y-1)}{q + Q(y)}W^{(q)}(y-1,y+k) \right. \\
			&+ \left.\sum_{1 \leq j \leq k-1}\frac{Q(y,y+j)}{q + Q(y)}W^{(q)}(y+j,y+k) - W^{(q)}(y,y+k) \right).
		\end{split}
		\nonumber
	\end{align}
	Since $W^{(q)}(y-1,y+k) > 0$ from (i) in \eqref{generalAssumptions}, we obtain
	\begin{align}
		qW^{(q)}(y,y+k) = \sum_{-1 \leq j < k} Q(y,y+j) W^{(q)}(y+j,y+k), \nonumber
	\end{align}
	which means
	\begin{align}
		Q W^{(q)}(x,y) = q W^{(q)}(x,y) \quad (x \geq 1, \ y \geq x+1). \label{eq13}
	\end{align}
	Since it holds $Q(x,x-k) = 0$ for $x \geq 2$ and $k \geq 2$ with $x-k \geq 0$,
	we see from \eqref{scaleFunc-offDiagonal} that for $k \geq 1$
	\begin{align}
		QW^{(q)}(x,x-k) &= \sum_{n \geq -1}Q(x,x+n) W^{(q)}(x+n,x-k) \nonumber \\
		&= \sum_{-1 \leq n < -k}Q(x,x+n)W^{(q)}(x+n,x-k) = 0. \label{eq11}
	\end{align}
	Hence, from \eqref{eq12}, \eqref{eq13} and \eqref{eq11}, the matrix $F = W^{(q)}$ is a solution of \eqref{eigenFunc}.

	Let $M$ be another solution of \eqref{eigenFunc} and, set $G := W^{(q)} - M$.
	Then $G$ is a triangular matrix satisfying
	\begin{align}
			Q G(x,y) = q G(x,y) \quad  (x \geq 1, y \geq 0). \nonumber
	\end{align}
	Suppose $ G(x,x+l) = 0 \ (x \geq 0, 0 \leq l \leq k)$ for some $k \geq 0$.
	Then it follows for $x \geq 1$
	\begin{align}
		0 = Q G(x,x+k) = Q(x,x-1) G(x-1,x+k). \nonumber
	\end{align}
	Since $Q(x,x-1) > 0 \ (x \geq 1)$ by \eqref{generalAssumptions},
	we obtain $G(x,x+k+1) = 0$ for $x \geq 0$.
	Thus, we see inductively that $G(x,y) = 0$ for every $x,y \geq 0$,
	which shows the solution of \eqref{eigenFunc} is unique.

	Set $M^{(q)}$ as the RHS of \eqref{polynomialRepresentation}.
	Note that since $W$ is triangular, it holds $W^{n}(x,y) = 0$ for $n > y-x$.
	From \eqref{eigenFunc} for $q = 0$, we see $F = M^{(q)}$ is a solution of \eqref{eigenFunc} for $q > 0$.
	Therefore, from the uniqueness of the solution, we obtain $W^{(q)} = M^{(q)}$.
	For $q \in \bC$, it is clear that $\sum_{n \geq 0}q^{n}W^{n} = (I - qW)^{-1}$.
	Thus, we see \eqref{q-scaleRepresentation} from \eqref{polynomialRepresentation}. 
\end{proof}

From the representation \eqref{q-scaleRepresentation}, we may derive the resolvent identity,
whose proof is clear and we omit it.

\begin{Cor} \label{cor:resolventEq}
	For $q,r \in \bC$ it holds
	\begin{align}
		W^{(q)} - W^{(r)} = (q-r) W^{(q)}W^{(r)} = (q-r) W^{(r)}W^{(q)}. \nonumber
	\end{align}
\end{Cor}

We also have a similar identity for $Z^{(q)}$.

\begin{Cor} \label{cor:resolventEqForZ}
	For $q,r \in \bC$ it holds
	\begin{align}
		Z^{(q)} - Z^{(r)} = (q-r) W^{(q)}Z^{(r)} = (q-r) W^{(r)}Z^{(q)}. \label{resolventEqForZ}
	\end{align}
\end{Cor}

\begin{proof}
	When $q = r = 0$, the assertion is obvious.
	For $q \neq 0$ and $r = 0$, from the definition of $Z^{(q)}$, it holds for $0 \leq x < y$ that
	\begin{align}
		Z^{(q)}(x,y) - Z^{(0)}(x,y) &= Z^{(q)}(x,y) - 1 = q \sum_{x < u < y} W^{(q)}(x,u) = q W^{(q)}Z^{(0)}(x,y). \nonumber
	\end{align}
	Let $r,q \neq 0$ and $x < y$.
	From the definition of $Z^{(q)}$, we have on the one hand, 
	\begin{align}
		\sum_{x < u < y}(W^{(q)}(x,u) - W^{(r)}(x,u)) 
		=\frac{Z^{(q)}(x,y)-1}{q} - \frac{Z^{(r)}(x,y)-1}{r}. \label{eq57}
	\end{align}
	On the other hand, it holds from Corollary \ref{cor:resolventEq} and Fubini's theorem that the LHS of \eqref{eq57} is equal to
	\begin{align}
		(q-r) \sum_{x < u < y} W^{(q)} W^{(r)}(x,u) &= (q-r) \sum_{x < v < y} W^{(q)}(x,v) \sum_{v < u < y} W^{(r)}(v,u)m(du) \nonumber \\
		&= \frac{q-r}{r}W^{(q)} Z^{(r)}(x,y) - \frac{q-r}{qr}(Z^{(q)}(x,y) - 1). \label{eq58}
	\end{align}
	From \eqref{eq57} and \eqref{eq58}, we obtain \eqref{resolventEqForZ}.
\end{proof}

We give an elementary estimate for the scale function.

\begin{Prop} \label{prop:estimateForW}
	For $n \geq 1$, it holds
	\begin{align}
		W^{n}(x,y) \leq \frac{W(x,y)}{(n-1)!} \left( \sum_{x < u < y} W(x,u) \right)^{n-1} \quad (0 \leq x < y), \label{eq23}
	\end{align}
	which implies for $q \in \bC$ 
	\begin{align}
		|W^{(q)}(x,y)| \leq W(x,y)\mathrm{e}^{|q|\sum_{x < u < y}W(x,u)} \quad (0 \leq x < y). \nonumber 
	\end{align}
\end{Prop}

\begin{proof}
	For $x \geq 0$, set $F_{x}(t) := \sum_{x < u \leq t} W(x,u) \ (t \in [0,\infty))$.
	Note that by integration by parts (see e.g., \cite[Lemma 5.13.1]{Ito_essentials}), it holds for $n \geq 2$
	\begin{align}
		d (F^{n}_{x})(t) = F_{x}(t)dF^{n-1}_{x}(t) + F_{x}(t-1)^{n-1}dF_{x}(t), \label{eq26}
	\end{align}
	and we see inductively the following inequality of measures:
	\begin{align}
		n F_{x}(t-1)^{n-1} dF_{x}(t) \leq d(F_{x}^{n})(t). \label{eq24}
	\end{align}
	We show \eqref{eq23} by induction.
	Suppose \eqref{eq23} holds for some $n \geq 1$.
	From \eqref{eq26} and \eqref{eq24} we have
	\begin{align}
		W^{n+1}(x,y) = &\sum_{x < u < y}W^{n}(x,u)W(u,y) \nonumber \\
		\leq &\frac{W(x,y)}{(n-1)!} \sum_{x < u < y}W(x,u)F_{x}(u-1)^{n-1} \nonumber \\
		= &\frac{W(x,y)}{(n-1)!} \int_{(x,y)} F_{x}(t-1)^{n-1}dF_{x}(t) \nonumber \\
		\leq & \frac{W(x,y)}{n!} F_{x}(y-1)^{n}. \nonumber
	\end{align}
\end{proof}

\section{Existence of quasi-stationary distributions} \label{section:QSD}

In this section, we always assume \eqref{nonExplosive} holds.
Define for $0 \leq x \leq y$
\begin{align}
	g^{(q)}(x,y) := \bE_{y}[\mathrm{e}^{-q\tau_{x}}] \quad (q \geq 0). \label{}
\end{align}
We especially write $g^{(q)}(x) := g^{(q)}(0,x)$.
Taking limit as $z \to \infty$ in \eqref{potentialDensity}, we have from \eqref{nonExplosive}
\begin{align}
	\begin{split}
		&\int_{0}^{\infty}\mathrm{e}^{-qt} \bP_{y}[X_{t} = u, \tau_{x} > t]dt \\
		= &g^{(q)}(x,y) W^{(q)}(x,u) - W^{(q)}(y,u) \quad (q \geq 0,\ y,u \geq x).
	\end{split}
	 \label{potentialDensityOneside}
\end{align}
For $x \geq 0$, define
\begin{align}
	\lambda^{[x]}_{0} := \sup \{ \lambda \geq 0 \mid \bE_{y}[\mathrm{e}^{\lambda \tau_{x}}] < \infty \quad \text{for every $y > x$} \}, \nonumber
\end{align}
where we note that $\lambda_{0} = \lambda_{0}^{[0]}$, where $\lambda_{0}$ was defined in \eqref{lambdaZero}, the function $x \mapsto \lambda_{0}^{[x]}$ is clearly non-decreasing and $\lambda_{0}^{[x]} < \infty$ from the regularity assumption (ii) of \eqref{generalAssumptions}.
From the analytic extension, the equality \eqref{potentialDensityOneside} can be extended to $q \in \bC$ with $\Re q > -\lambda_{0}^{[x]}$, where $\Re z$ denotes the real part of $z \in \bC$.

The proofs of Theorems \ref{thm:QSDExistenceCharacterization} and \ref{thm:stochasticOrdering}, \ref{thm:spectralBottomCharacterization} are given by compiling auxiliary propositions we show below.

Every quasi-stationary distribution is represented by a scale function. 

\begin{Lem} \label{lem:qsdCharacterization}
	Let $\nu$ be a quasi-stationary distribution such that $\bP_{\nu}[\tau_{0} > t] = \mathrm{e}^{-\lambda t}$ for some $\lambda > 0$.
	Then it holds
	\begin{align}
		\nu(x) = \lambda W^{(-\lambda)}(0,x) \quad (x \geq 0). \nonumber
	\end{align}
\end{Lem}

\begin{proof}
	On the one hand, from \eqref{potentialDensityOneside} for $q = 0$, it holds for $y \geq 0$
	\begin{align}
		\int_{0}^{\infty}\bP_{\nu}[X_{t} = y, \tau_{0} > t]dt = W(0,y) - \sum_{x \geq 0}\nu(x) W(x,y). \nonumber
	\end{align}
	On the other hand, since $\nu$ is a quasi-stationary distribution, we have
	\begin{align}
		\int_{0}^{\infty}\bP_{\nu}[X_{t} = y, \tau_{0} > t]dt = \nu(y)\int_{0}^{\infty}\mathrm{e}^{-\lambda t}dt = \frac{\nu(y)}{\lambda}. \nonumber
	\end{align}
	Thus, we see that
	\begin{align}
		\nu(y) = \lambda W(0,y) - \lambda \sum_{x \geq 0}\nu(x) W(x,y), \nonumber
	\end{align}
	or regarding $\nu$ as a row vector,
	\begin{align}
		\nu (I + \lambda W) = \lambda W(0,\cdot). \nonumber
	\end{align}
	Hence, we have $\nu = \lambda W  (I + \lambda W)^{-1}(0,\cdot) = \lambda W^{(-\lambda)}(0,\cdot)$ from \eqref{q-scaleRepresentation}.
\end{proof}

The following is a key proposition.
We introduce a function $h^{(q)}$ which is useful to characterize existence of a quasi-stationary distribution.

\begin{Prop} \label{prop:nonnegativityImpliesIntegrability-entrance}
	For $x \geq 0$, let $\lambda \in (-\infty,\lambda_{0}^{[x]}]$.
	For $q > -\lambda$, the limit
	\begin{align}
		h^{(q)}(\lambda;x) := \lim_{y \to \infty} \frac{W^{(-\lambda)}(x,y)}{W^{(q)}(x,y)} \label{defOfH}
	\end{align}
	exists (we especially write $h := h^{(0)}$) and finite.
	In addition, it holds
	\begin{align}
		\sum_{y > x} |W^{(-\lambda)}(x,y)| g^{(q)}(x,y)  < \infty \label{eq44}
	\end{align}
	and
	\begin{align}
		h^{(q)}(\lambda;x) = 1 - (\lambda+q) \sum_{y > x}W^{(-\lambda)}(x,y)  g^{(q)}(x,y) . \label{h-funcRepresentation}
	\end{align}
	The function $h^{(q)}$ has the following properties:
	\begin{enumerate}
		\item $0 \leq h^{(q)}(\lambda;x) < 1$.
		\item $h^{(q)}(\lambda;x)$ is non-increasing in $\lambda$ and $q$.
		\item For fixed $q > 0$, the RHS of \eqref{h-funcRepresentation} can be analytically extended in $\lambda \in \bC$ with $|\lambda| < q$.
		\item For fixed $\lambda \in (-\infty,\lambda_{0}^{[x]})$, the RHS of \eqref{h-funcRepresentation} can be analytically extended in $q \in \bC$ with $\Re q > -\lambda$.
		\item If $h^{(q)}(\lambda;x) = 0$ for some $\lambda \in (-\infty, \lambda_{0}^{[x]})$ and $q \in (-\lambda,\infty)$,
		it also holds for every $\lambda \in (-\infty, \lambda_{0}^{[x]}]$ and $q \in (-\lambda,\infty)$.
	\end{enumerate}
\end{Prop}

\begin{proof}
	Fix $x \geq 0$.
	From the definition of $\lambda_{0}^{[x]}$, it clearly holds for $\lambda \in (-\infty,\lambda_{0}^{[x]})$
	\begin{align}
		W^{(-\lambda)}(x,y) = \frac{1}{n_{y}[\mathrm{e}^{\lambda \tau_{x}}, \tau_{x} < \infty]} > 0 \quad (0 \leq x < y). \label{eq54}
	\end{align}
	We also see $W^{(-\lambda_{0}^{[x]})}(x,y) \geq 0$ by the continuity.
	Take $\lambda \in (-\infty,\lambda_{0}^{[x]}]$ and $q  >-\lambda$.
	From Corollary \ref{cor:resolventEq}, 
	\begin{align}
		W^{(-\lambda)}(x,y) = W^{(q)}(x,y) - (\lambda+q) \sum_{x < u < y}W^{(-\lambda)}(x,u)W^{(q)}(u,y). \nonumber
	\end{align}
	From Proposition \ref{prop:exitProblem1}, it holds for $x > 0$
	\begin{align}
		0 \leq \frac{W^{(-\lambda)}(x,y)}{W^{(q)}(x,y)} = 1 - (\lambda+q) \sum_{x < u < y} W^{(-\lambda)}(x,u) \bE_{u}[\mathrm{e}^{-q\tau_{x}},\tau_{x} < \tau_{y}^{+}] . \nonumber
	\end{align}
	Taking the limit as $y \to \infty$, we have from the monotone convergence theorem
	\begin{align}
		0 \leq \lim_{y \to \infty} \frac{W^{(-\lambda)}(x,y)}{W^{(q)}(x,y)} = 1 - (\lambda+q) \sum_{u > x} W^{(-\lambda)}(x,u) g^{(q)}(x,u)  < 1. \label{eq18-entrance}
	\end{align}
	Thus, we obtain \eqref{eq44} and \eqref{h-funcRepresentation}.
	The properties (i) and (ii) are obvious.
	We show (iii).
	Take $q > 0$ and $0 < r < q $. Let $\zeta \in \bC$ with $|\zeta| \leq r$.
	Since it holds $|W^{(-\zeta)}(x,y)| \leq W^{(r)}(x,y)$ from \eqref{polynomialRepresentation} and $\sum_{u > x} W^{(r)}(x,u) g^{(q)}(x,u) \leq (q-r)^{-1}$
	from \eqref{eq18-entrance},
	we see the function $\zeta \mapsto \sum_{u > x} W^{(-\zeta)}(x,u) g^{(q)}(x,u)$ is analytic on $|\zeta| < q$ by Montel's theorem.
	The property (iv) follows similarly since it holds for $q > -\lambda$ and $s \in \bR$ that $|g^{(q+is)}(x,u)| \leq g^{(q)}(x,u)$ and $\sum_{z > x} W^{(-\lambda)}(x,u) g^{(q)}(x,u) \leq (q + \lambda)^{-1}$.
	We finally show (v).
	Suppose $h^{(q')}(\lambda';x) = 0$ for some $\lambda' \in (-\infty,\lambda_{0}^{[x]})$ and $q' \in (-\lambda',\infty)$.
	From (ii), the function $h^{(q')}(\lambda;x)$ is non-increasing in $\lambda$, and it follows $h^{(q')}(\lambda;x) = 0$ for $\lambda \in [\lambda',\lambda_{0}^{[x]}]$.
	Take $r > |\lambda'| \vee q'$.
	Since $h^{(r)}(\lambda;x) = 0$ for $\lambda \in [\lambda', \lambda_{0}^{[x]}]$, it follows from (iii) and the identity theorem that $h^{(r)}(\lambda;x) = 0$ for $\lambda \in (-r,\lambda_{0}^{[x]}]$.
	Take $\lambda \in (-r,\lambda_{0}^{[x]}]$ arbitrarily.
	Since $h^{(q)}(\lambda;x) = 0$ for $q \in [r,\infty)$, it follows from (iv) and the identity theorem that $h^{(q)}(\lambda;x) = 0$ for $q \in (-\lambda,\infty)$.
	Since $r$ can be taken arbitrarily large, we obtain (v).
\end{proof}

From \eqref{h-funcRepresentation} for $q = 0$, we have the integrability of $W^{(-\lambda
)}(x,\cdot) \ (\lambda \in (0,\lambda_{0}^{[x]}])$.

\begin{Cor} \label{cor:ScaleDefinesFiniteMeasure}
	Let $x \geq 0$ and suppose $\lambda_{0}^{[x]} > 0$.
	For $\lambda \in (0, \lambda_{0}^{[x]}]$ it holds
	\begin{align}
		\lambda \sum_{y > x} W^{(-\lambda)}(x,y) = 1 - h (\lambda;x) \in (0,1]. \nonumber
	\end{align}
	It also holds $h(\lambda;x) = \lim_{y \to \infty} Z^{(-\lambda)}(x,y)$.
\end{Cor}

As another corollary, we show the positivity of $W^{(-\lambda_{0})}(0,\cdot)$.

\begin{Cor} \label{cor:non-negativityOfScale-entrance}
	It holds $W^{(-\lambda_{0})}(0,x) > 0$ for $x > 0$.
\end{Cor}

\begin{proof}
	If $\lambda_{0} = 0$, the assertion is obvious.
	Suppose $\lambda_{0} > 0$.
	From Corollary \ref{cor:resolventEq} and Corollary \ref{cor:ScaleDefinesFiniteMeasure},
	we have for $x \geq 2$
	\begin{align}
		W^{(-\lambda_{0})}(0,x) &= W(0,x) - \lambda_{0} \sum_{0 < u < x}W^{(-\lambda_{0})}(0,u)W(u,x) \nonumber \\
		&= \frac{\lambda_{0}W(0,x)}{1 - h(\lambda_{0};0)}\sum_{u > 0}W^{(-\lambda_{0})}(0,u) - \lambda_{0} \sum_{0 < u < x}W^{(-\lambda_{0})}(0,u)W(u,x) \nonumber \\
		&= \frac{\lambda_{0} W(0,x)}{1 - h(\lambda_{0};0)} \sum_{u > 0} W^{(-\lambda_{0})}(0,u)
		\left( 1 - (1 - h(\lambda_{0};0)) \bP_{u}[\tau_{0} < \tau_{x}^{+}] \right) \nonumber \\
		&\geq \frac{\lambda_{0} W(0,x)W(0,1)}{1 - h(\lambda_{0};0)} 
		\left( 1 - (1 - h(\lambda_{0};0)) \bP_{1}[\tau_{0} < \tau_{x}^{+}] \right) \nonumber \\
		&> 0, \nonumber 
	\end{align}
	where we note that $W^{(-\lambda_{0})}(0,1) = W(0,1) > 0$ and $\bP_{1}[\tau_{0} < \tau_{x}^{+}] < 1$ from (i) of \eqref{generalAssumptions}.
\end{proof}

The following lemma shows that existence of a quasi-stationary distribution is characterized whether $h^{(q)}(\lambda;0)$ is zero or not.

\begin{Lem} \label{lem:scaleFuncGivesQSD-entrance}
	Suppose $\lambda_{0} > 0$.
	For $0 < \lambda \leq \lambda_{0}$, the (sub)probability distribution
	\begin{align}
		\nu_{\lambda}(x) := \lambda W^{(-\lambda)}(0,x) \quad (x \geq 0) \nonumber
	\end{align}
	is a quasi-stationary distribution with $\bP_{\nu_{\lambda}}[\tau_{0} > t] = \mathrm{e}^{-\lambda t} \ (t \geq 0)$ if and only if $h^{(q)}(\lambda;0) = 0$ for some $q \geq 0$.
	It is also equivalent to 
	\begin{align}
		1 = \sum_{y > 0}\nu_{\lambda}(y). \label{eq19}
	\end{align}
\end{Lem}

\begin{proof}
	Consider the Laplace transform of $\sum_{x > 0} \nu_{\lambda}(x)\bP_{x}[X_{t} = y, \tau_{0} > t] \ (y > 0)$.
	For $q \geq 0$, it holds from Corollary \ref{cor:resolventEq}, \eqref{potentialDensityOneside} and \eqref{h-funcRepresentation}
	\begin{align}
		& \int_{0}^{\infty}\mathrm{e}^{- qt} \sum_{x \geq 0} \nu_{\lambda}(x) \bP_{x}[X_{t} = y, \tau_{0} > t]dt \nonumber \\
		=&\lambda \sum_{x \geq 0} W^{(-\lambda)}(0,x)(g^{(q)}(x)W^{(q)}(0,y) - W^{(q)}(x,y)) \nonumber \\
		=&\lambda \left(\sum_{x \geq 0} W^{(-\lambda)}(0,x)g^{(q)}(x)\right)W^{(q)}(0,y) - \lambda  W^{(-\lambda)} W^{(q)}(0,y) \nonumber \\
		=& \frac{\lambda}{\lambda + q}\left( (1 - h^{(q)}(\lambda;0))W^{(q)}(0,y) - (W^{(q)}(0,y) - W^{(-\lambda)}(0,y)) \right) \nonumber \\
		=& \frac{\lambda W^{(-\lambda)}(0,y) }{\lambda + q} - \frac{\lambda}{\lambda + q} h^{(q)}(\lambda;0) W^{(q)}(0,y). \nonumber 
	\end{align}
	Since it holds 
	\begin{align}
		\frac{\lambda W^{(-\lambda)}(0,y) }{\lambda + q} = \int_{0}^{\infty}\mathrm{e}^{-qt} (\mathrm{e}^{-\lambda t}\nu_{\lambda}(y)) dt, \nonumber
	\end{align}
	we see from Proposition \ref{prop:nonnegativityImpliesIntegrability-entrance} (v) that the measure $\nu_{\lambda}$ is a quasi-stationary distribution if and only if $h^{(q)}(\lambda;0) = 0$ for some $q \geq 0$.
	The equivalence with \eqref{eq19} is clear from Corollary \ref{cor:ScaleDefinesFiniteMeasure}.
\end{proof}

From Lemmas \ref{lem:qsdCharacterization} and \ref{lem:scaleFuncGivesQSD-entrance},
to see existence of a quasi-stationary distribution, it is enough to consider the root of the function $h$.
The following Lemma shows that the function $h$ is largely determined by the boundary classification.

\begin{Lem} \label{lem:BdryClassificationEquivalence}
	Suppose $\lambda_{0} > 0$. 
	The boundary $\infty$ is non-entrance if and only if $h(\lambda;0) = 0$ for some (or equivalently every) $0 < \lambda < \lambda_{0}$.
\end{Lem}

\begin{proof}
	Suppose the boundary $\infty$ is non-entrance.
	Take $\lambda \in (0,\lambda_{0})$.
	If $h(\lambda;0) \neq 0$, it follows from \eqref{defOfH} for $q = 0$
	\begin{align}
		\lim_{y \to \infty}\frac{W^{(-\lambda)}(0,y)}{W(0,y)} \in (0,1). \nonumber
	\end{align}
	Since $\sum_{y \geq 0}W^{(-\lambda)}(0,y) < \infty$ from Corollary \ref{cor:ScaleDefinesFiniteMeasure},
	it follows $\sum_{y \geq 0}W(0,y) < \infty$ and contradicts to the non-entrance condition.

	Suppose $h(\lambda';0) = 0$ for some $\lambda' \in (0,\lambda_{0})$.
	From Proposition \ref{prop:nonnegativityImpliesIntegrability-entrance} (v), it also holds for every $\lambda \in (0,\lambda_{0})$.
	From Corollary \ref{cor:resolventEq} and Lemma \ref{lem:scaleFuncGivesQSD-entrance} we have
	\begin{align}
		\frac{1}{\lambda} = \sum_{y \geq 0}W^{(-\lambda)}(0,y) \leq \sum_{y \geq 0}W(0,y). \nonumber
	\end{align}
	Taking limit as $\lambda \to 0+$, we obtain the non-entrance condition.
\end{proof}

From Lemmas \ref{lem:qsdCharacterization}, \ref{lem:scaleFuncGivesQSD-entrance} and \ref{lem:BdryClassificationEquivalence},
we easily see that Theorem \ref{thm:QSDExistenceCharacterization} (ii), the non-entrance boundary case, holds.

We prove Theorem \ref{thm:stochasticOrdering}.

\begin{proof}[Proof of Theorem \ref{thm:stochasticOrdering}]
	It is enough to show the function
	\begin{align}
		f_{x}(\lambda) :=  \frac{W^{(-\lambda)}(0,x)}{W^{(-\lambda)}(0,x+1)} \quad (\lambda \in (0,\lambda_{0}]) \nonumber
	\end{align}
	is strictly increasing for every $x \geq 1$.
	Indeed, if it holds, we have for $x \geq 1$ and $0 < \lambda' < \lambda \leq \lambda_{0}$
	\begin{align}
		\nu_{\lambda}(x)\nu_{\lambda'}(x+1) - \nu_{\lambda'}(x)\nu_{\lambda}(x+1) \nonumber 
		= \nu_{\lambda}(x+1) \nu_{\lambda'}(x+1) (f_{x}(\lambda) - f_{x}(\lambda')) > 0. \nonumber
	\end{align}
	From Theorem \ref{thm:potentialDensity}, we have for $q \geq 0$
	\begin{align}
		\bE_{x}\left[ \int_{0}^{\tau_{0} \wedge \tau_{x+1}^{+}} \mathrm{e}^{-q t} dL_{t}^{x} \right] = \frac{W^{(q)}(0,x)W(x,x+1)}{W^{(q)}(0,x+1)}, \label{}
	\end{align}
	where we note that $W^{(q)}(x,x+1) = W(x,x+1)$.
	Then we have
	\begin{align}
		f_{x}(\lambda) = \frac{1}{W(x,x+1)} \bE_{x}\left[ \int_{0}^{\tau_{0} \wedge \tau_{x+1}^{+}} \mathrm{e}^{\lambda t} dL_{t}^{x} \right] \quad (\lambda \in (0,\lambda_{0}]) \nonumber
	\end{align}
	by the analytic extension and the monotone convergence theorem,
	which shows $f_{x}$ is strictly increasing.
\end{proof}

In order to show Theorem \ref{thm:QSDExistenceCharacterization} (i), we focus on the entrance boundary case.
Under the entrance condition, the function $W^{(q)}(x,\cdot)$ is integrable for every $q \in \bC$.

\begin{Lem} \label{lem:analyticExtentionOfh}
	Suppose the boundary $\infty$ is entrance.
	Then it holds for every $R > 0$
	\begin{align}
		\sup_{\zeta \in \bC, |\zeta| \leq R}\sum_{y > x} |W^{(\zeta)}(x,y)| < \infty \quad \text{for every $x \geq 0$}. \label{eq21}
	\end{align}
	In addition, the limit $Z^{(q)}(x) := \lim_{y \to \infty}Z^{(q)}(x,y) \ (q \in \bC, x \geq 0)$
	exists and
	\begin{align}
		Z^{(q)}(x) = 1 + q \sum_{y > x} W^{(q)}(x,y) \quad (q \in \bC, \ x \geq 0). \label{eq50}
	\end{align}
	Hence, the function $q \mapsto Z^{(q)}(x) \ (x \geq 0)$ is an entire function.
	Moreover, it holds
	\begin{align}
		g^{(q)}(x,y) = \frac{Z^{(q)}(y)}{Z^{(q)}(x)} \quad (q > -\lambda_{0}^{[x]},\  y \geq x), \label{eq53}
	\end{align}
	and thus the function $q \mapsto g^{(q)}(x,y) \ (y \geq x)$ is analytically extended to a meromorphic function on $\bC$.
\end{Lem}

\begin{proof}
	From Proposition \ref{prop:estimateForW} and the entrance condition,
	we easily see \eqref{eq21} and \eqref{eq50} hold.
	From \eqref{eq51} it holds for $0 \leq x < y < z$ and $q \geq 0$
	\begin{align}
		\frac{Z^{(q)}(y,z)}{Z^{(q)}(x,z)} = \frac{W^{(q)}(y,z)}{W^{(q)}(x,z)} + \frac{\bE_{y}[\mathrm{e}^{-q\tau_{z}^{+}},\tau_{z}^{+} < \tau_{x}]}{Z^{(q)}(x,z)}. \nonumber
	\end{align}
	Since it holds $\lim_{z \to \infty}\bE_{y}[\mathrm{e}^{-q\tau_{z}^{+}},\tau_{z}^{+} < \tau_{x}] = 0$ from \eqref{nonExplosive}, we have
	\begin{align}
		g^{(q)}(x,y) = \lim_{z \to \infty} \frac{W^{(q)}(y,z)}{W^{(q)}(x,z)} = \lim_{z \to \infty} \frac{Z^{(q)}(y,z)}{Z^{(q)}(x,z)} = \frac{Z^{(q)}(y)}{Z^{(q)}(x)}. \label{eq55}
	\end{align}
	From Corollary \ref{cor:ScaleDefinesFiniteMeasure} and Lemma \ref{lem:BdryClassificationEquivalence} it holds that 
	\begin{align}
		Z^{(q)}(x) = h(-q;x) > 0 \quad (q \in (-\lambda_{0}^{[x]},0)) \label{eq56}
	\end{align}
	when $\lambda_{0}^{[x]} > 0$,
	and the equality \eqref{eq53} is analytically extended to $q \in \bC$ with $\Re q > -\lambda_{0}^{[x]}$.
\end{proof}

For $0 \leq x < y$, define 
\begin{align}
	\lambda_{0}^{[x]}(y) := \sup \{ \lambda \geq 0 \mid \bE_{y}[\mathrm{e}^{\lambda \tau_{x} }] < \infty \}. \label{}
\end{align}	

\begin{Rem} \label{rem:remarkOnLambdaXY}
	It holds $\lambda_{0}^{[0]}(y) = \lambda_{0} \ (y > 0)$ by the irreducibility condition (i) in \eqref{generalAssumptions}.
	It also holds from the downward skip-free property that the function $\{x+1,x+2,\cdots\} \ni y \mapsto \lambda_{0}^{[x]}(y)$ is non-increasing, and that $\lambda_{0}^{[x]}(y) \leq \inf_{x < u \leq y} Q(u) < \infty \ (x < y)$, where we note the obvious equality $\bP_{y}[\tau_{\bN \setminus \{y\} } \geq t] = \mathrm{e}^{-Q(y)t} \ (t \geq 0)$.
\end{Rem}

The following lemma and its corollary show that under the entrance condition, it always holds $\lambda_{0} > 0$ and $h(\lambda_{0};0) = 0$.
From Lemmas \ref{lem:scaleFuncGivesQSD-entrance} and \ref{lem:BdryClassificationEquivalence}, that implies there exists a unique quasi-stationary distribution $\nu_{\lambda_{0}}$ and completes the proof of Theorem \ref{thm:QSDExistenceCharacterization} (i).
The following proof owes to \cite[Proposition 3]{BertoinQSD}.

\begin{Lem} \label{lem:r-recurrence}
	Assume the boundary $\infty$ is entrance.
	Let $0 \leq x < y$.
	Then it holds that $\lambda_{0}^{[x]}(y) > 0$ and that
	the function $\bC \ni q \mapsto Z^{(q)}(y) / Z^{(q)}(x)$ has a pole at $q = -\lambda_{0}^{[x]}(y)$, which is a pole with the minimum absolute value.
	In particular, it holds
	\begin{align}
		\lim_{q \to -\lambda_{0}^{[x]}(y)+}g^{(q)}(x,y) = \infty. \label{}
	\end{align}
\end{Lem}

\begin{proof}
	Take $y > x \geq 0$.
	Since the LHS of \eqref{eq53} is analytic for $q \in \{ \zeta \in \bC \mid \Re \zeta >
	 -\lambda_{0}^{[x]}(y) \}$,
	it holds 
	\begin{align}
		\text{the function $Z^{(q)}(y) / Z^{(q)}(x)$ does not have a pole at $q \in \{ \zeta \in \bC \mid \Re \zeta > - \lambda_{0}^{[x]}(y) \}$}. \label{eq35-c}
	\end{align}
	Suppose there exists a pole of the function $\bC \ni q \mapsto Z^{(q)}(y) / Z^{(q)}(x)$ and let $\zeta'$ be the one with the minimum absolute value.
	Note that $\zeta' \neq 0$ and thus $|\zeta'| > 0$ since $Z^{(0)}(y) / Z^{(0)}(x) = 1$.
	We show $\zeta' = -\lambda_{0}^{[x]}(y)$.
	From \eqref{eq35-c}, we see that $\Re \zeta' \leq -\lambda_{0}^{[x]}(y)$, which implies $|\zeta'| \geq \lambda_{0}^{[x]}(y)$.
	Suppose $|\zeta'| > \lambda_{0}^{[x]}(y)$.
	Then the function $Z^{(q)}(y) / Z^{(q)}(x)$ has a power series expansion around $0$ for $\zeta \in \bC$ with $|\zeta| < |\zeta'|$ and
	the coefficient of $\zeta^{n}$ is given by the $n$-th right-derivative of $g^{(q)}(x,y)$ at $\zeta = 0$ divided by $n!$, i.e., $(-1)^{n}\bE_{y}[\tau_{x}^{n}] / n!$.
	Since the series absolutely converges for $\zeta \in \bC$ with $|\zeta| < |\zeta'|$, it follows $\bE_{y}[\mathrm{e}^{(|\zeta'| - \delta)\tau_{x}}] < \infty$ for every $\delta > 0$.
	It contradicts to the definition of $\lambda_{0}^{[x]}(y)$, and thus it follows $\zeta' = -\lambda_{0}^{[x]}(y)$.
	When we suppose there are no poles of the function $Z^{(q)}(y) / Z^{(q)}(x)$,
	by the same argument above it follows $\lambda_{0}^{[x]}(y) = \infty$ and it is impossible from Remark \ref{rem:remarkOnLambdaXY}.
\end{proof}

The following corollary easily follows from Lemma \ref{lem:r-recurrence} for $x = 0$, Remark \ref{rem:remarkOnLambdaXY} and \eqref{eq56}. 

\begin{Cor}
	When the boundary $\infty$ is entrance, it holds $\lambda_{0} > 0$, $h(\lambda_{0};0) = 0$ and 
	\begin{align}
		\lambda_{0} = \min \{ \lambda \geq 0 \mid Z^{(-\lambda)}(0) = 0 \}. \label{eq45}
	\end{align}
\end{Cor}

We prove Theorem \ref{thm:spectralBottomCharacterization}.

\begin{proof}[Proof of Theorem \ref{thm:spectralBottomCharacterization}]
	Assume there exists $\lambda' > \lambda_{0}$ such that $W^{(-\lambda')}(0,x) > 0$ for every $x > 0$.
	From Corollary \ref{cor:non-negativityOfScale-entrance}, it is enough to show that this assumption leads to a contradiction.
	From the same argument in Proposition \ref{prop:nonnegativityImpliesIntegrability-entrance},
	the limit $h^{(q)}(\lambda';0) :=  \lim_{y \to\infty} W^{(-\lambda')}(0,y) / W^{(q)}(0,y)$ exists for $q \in (-\lambda_{0},\infty) \cup \{0\}$ and it holds
	\begin{align}
		h^{(q)}(\lambda';0) = 1 - (\lambda' + q) \sum_{u > 0} W^{(-\lambda')}(0,u)g^{(q)}(u) \in [0,1). \label{eq43}
	\end{align}
	Note that this implies $\lambda'\sum_{u > 0}W^{(-\lambda')}(0,u) \in (0,1]$.
	Let the boundary $\infty$ be non-entrance.
	If $h^{(q)}(\lambda';0) > 0$ for some $q \geq 0$,
	it follows from the same argument in Lemma \ref{lem:BdryClassificationEquivalence} that $\sum_{u > 0}W^{(q)}(0,u) < \infty$ and that contradicts to the non-entrance condition.
	Thus, we have $h^{(q)}(\lambda';0) = 0$ for $q \geq 0$.
	Then the same computation in Lemma \ref{lem:qsdCharacterization} shows that $\nu_{\lambda'}(x):= \lambda' W^{(-\lambda')}(0,x) \ (x \geq 0)$ is a quasi-stationary distribution such that $\bP_{\nu_{\lambda'}}[\tau_{0} > t] = \mathrm{e}^{-\lambda' t}$, which is, however, impossible from the definition of $\lambda_{0}$.
	Let the boundary $\infty$ be entrance.
	Taking limit as $q \to -\lambda_{0}+$ in \eqref{eq43}, we see from Lemma \ref{lem:r-recurrence} that $\lim_{q \to -\lambda_{0}+}h^{(q)}(\lambda) = -\infty$ and it contradicts to the non-negativity of $h^{(q)}(0;\lambda)$.
\end{proof}

\section{Yaglom limit for the entrance boundary case} \label{section:YaglomLimit}

When the boundary $\infty$ is entrance, 
we may further show the \textit{$\lambda_{0}$-positive recurrence}, which enables us to apply the $R$-theory for Markov chains (see, e.g., Kingman \cite{KingmanRrecurrence} and Anderson \cite[Chapter 5.2]{Anderson}) and derive the Yaglom limit. 
Also in this section, we always assume \eqref{nonExplosive}.

We first show the following.

\begin{Lem} \label{lem:lambdaX}
	Suppose the boundary $\infty$ is entrance.
	Then the following holds:
	\begin{enumerate}
		\item For every $x > 0$, it holds $Z^{(-\lambda_{0})}(x) \in (0,1)$ and the function $\bN \ni x \mapsto Z^{(-\lambda_{0})}(x)$ is strictly increasing and satisfies $\lim_{x \to \infty} Z^{(-\lambda_{0})}(x) = 1$.
		\item The function $\bN \ni x \mapsto Z^{(-\lambda_{0})}(x)$ is $\lambda_{0}$-invariant, that is,
		\begin{align}
			\bE_{x}[Z^{(-\lambda_{0})}(X_{t}),\tau_{0} > t] = \mathrm{e}^{-\lambda_{0}t}Z^{(-\lambda_{0})}(x) \quad (x \geq 0). \nonumber
		\end{align}
		\item For every $0 < x < y$, it holds $\lambda_{0}^{[x]}(y) > \lambda_{0}$.
		\item The function $\bC \ni q \mapsto Z^{(q)}(0)$ has a simple root at $q = -\lambda_{0}$ and
		\begin{align}
			\rho := \left. \frac{d}{dq} Z^{(q)}(0) \right|_{q = -\lambda_{0}} = \sum_{u > 0}W^{(-\lambda_{0})}(0,u)Z^{(-\lambda_{0})}(u) \in (0,\infty). \label{}
		\end{align}
	\end{enumerate}
\end{Lem}

\begin{proof}
	From Proposition \ref{prop:nonnegativityImpliesIntegrability-entrance}, Corollary \ref{cor:ScaleDefinesFiniteMeasure} and the inequality $\lambda_{0}^{[x]} \geq \lambda_{0} \ (x > 0)$, we see $Z^{(-\lambda_{0})}(x) \in [0,1)$.
	From \eqref{potentialDensityOneside}, \eqref{eq53}, \eqref{eq45} and Corollary \ref{cor:resolventEqForZ},
	we have for $q \in (-\lambda_{0},\infty)$ and $x > 0$
	\begin{align}
		&\int_{0}^{\infty}\mathrm{e}^{-qt} \bE_{x}[Z^{(-\lambda_{0})}(X_{t}), \tau_{0} > t] dt \label{eq64} \\
		= &\sum_{u > 0} \left( \frac{Z^{(q)}(x)}{Z^{(q)}(0)}W^{(q)}(0,u) - W^{(q)}(x,u) \right) Z^{(-\lambda_{0})}(u) \label{} \\
		= & \lim_{z \to \infty} \left( \frac{Z^{(q)}(x)}{Z^{(q)}(0)} W^{(q)} Z^{(-\lambda_{0})}(0,z) - W^{(q)} Z^{(-\lambda_{0})}(x,z)\right) \label{} \\
		= & \frac{1}{q + \lambda_{0}} \left( \frac{Z^{(q)}(x)}{Z^{(q)}(0)}(Z^{(q)}(0) - Z^{(-\lambda_{0})}(0)) - (Z^{(q)}(x) - Z^{(-\lambda_{0})}(x)) \right) \label{} \\
		= & \frac{Z^{(-\lambda_{0})}(x)}{q + \lambda_{0}}. \label{} 
	\end{align}
	Since the process $X$ is right-continuous and the function $I \ni x \mapsto Z^{(-\lambda_{0})}(x)$ is bounded,
	we see (ii) holds.
	If $Z^{(-\lambda_{0})}(x) = 0$ for some $x > 0$, it follows from \eqref{generalAssumptions} and \eqref{eq64} that $Z^{(-\lambda_{0})}(x) = 0$ every $x > 0$.
	It is, however, impossible because $\lim_{x \to \infty} Z^{(-\lambda_{0})}(x) = 1$,
	which can be seen from $\sum_{y > x} W^{(-\lambda_{0})}(x,y) \leq \sum_{y > x}W(0,y) \to 0 \ (x \to \infty)$.
	Thus, we obtain $Z^{(-\lambda_{0})}(x) \in (0,1) \ (x > 0)$.
	Since $Z^{(-\lambda)}(x) > 0 \ (\lambda \in [0,\lambda_{0}], x > 0)$,
	the assertion (iii) follows from Lemma \ref{lem:r-recurrence}.
	For $y > x$, it follows from \eqref{eq53} and (iii)
	\begin{align}
		Z^{(-\lambda_{0})}(y) = \bE_{y}[\mathrm{e}^{\lambda_{0} \tau_{x}}] Z^{(-\lambda_{0})}(x) > Z^{(-\lambda_{0})}(x), \label{}
	\end{align}
	and (i) is shown.
	From Corollary \ref{cor:resolventEq} and Fubini's theorem, we have
	\begin{align}
		Z^{(q)}(0) = Z^{(q)}(0) - Z^{(-\lambda_{0})}(0) = (q - (-\lambda_{0})) \sum_{u > 0}W^{(q)}(0,u)Z^{(-\lambda_{0})}(u). \label{}
	\end{align}
	Thus, it follows
	\begin{align}
		\left. \frac{d}{dq} Z^{(q)}(0) \right|_{q = -\lambda_{0}} = \sum_{u > 0}W^{(-\lambda_{0})}(0,u)Z^{(-\lambda_{0})}(u) > 0, \label{}
	\end{align}
	and (iv) holds.
\end{proof}

The following is the main theorem in this section, which obviously implies Theorem \ref{thm:YaglomSimplified}.

\begin{Thm} \label{thm:YaglomLimit}
	Suppose the boundary $\infty$ is entrance.
	Then the following holds: 
	\begin{enumerate}
		\item The process $X$ is $\lambda_{0}$-positive recurrent, that is, the limit $\lim_{t \to \infty} \mathrm{e}^{\lambda_{0}t}\bP_{x}[X_{t} = x, \tau_{0} > t]$ exists and positive for every $x > 0$. More precisely, it holds
		\begin{align}
			\lim_{t\to \infty} \mathrm{e}^{\lambda_{0}t}\bP_{x}[X_{t} = y,\tau_{0} > t] = \frac{Z^{(-\lambda_{0})}(x)W^{(-\lambda_{0})}(0,y)}{\rho}  \quad (x,y > 0). \nonumber
		\end{align}
		\item It holds for $x \geq 0$
		\begin{align}
			\lim_{t \to \infty} \mathrm{e}^{\lambda_{0}t}\bP_{x}[\tau_{0} > t] = \frac{ Z^{(-\lambda_{0})}(x)}{\rho\lambda_{0}}. \nonumber
		\end{align}
		\item The quasi-stationary distribution $\nu_{\lambda_{0}}$ is the Yaglom limit:
		\begin{align}
			\lim_{t \to \infty} \bP_{x}[X_{t} = y \mid \tau_{0} > t] = \nu_{\lambda_{0}}(y). \quad (x,y > 0). \nonumber
		\end{align}
	\end{enumerate}
\end{Thm}

\begin{proof}
	From \eqref{potentialDensityOneside}, Lemma \ref{lem:lambdaX} (i), (iv) and Karamata's Tauberian theorem \cite[Theorem 1.7.1]{Regularvariation}, we have
	\begin{align}
		\lim_{t \to \infty}\frac{1}{t}\int_{0}^{t} \mathrm{e}^{\lambda_{0}s} \bP_{x}[X_{s} = y, \tau_{0} > s] ds = \frac{Z^{(-\lambda_{0})}(x)W^{(-\lambda_{0})}(0,y)}{\rho} \quad (x,y > 0). \label{}
	\end{align}
	From \cite[p.177]{Anderson}, the limit $\lim_{t \to \infty}\mathrm{e}^{\lambda_{0}t}\bP_{x}[X_{t} = y, \tau_{0} > t]$ exists and thus (i) holds.
	The assertions (ii) and (iii) follows from \cite[Proposition 2.9, Chapter 5]{Anderson}.
\end{proof}
\appendix

\section{Infinity as an instantaneous entrance boundary}

In this appendix, we assume the condition \eqref{nonExplosive} holds.
Set $\bar{\bN} := \bN \cup \{\infty\}$ and equip $\bar{\bN}$ with the topology induced from the one-point compactification.
When the boundary $\infty$ is entrance, following Foucart, Li and Zhou \cite{entranceBdry} we can add $\infty$ to the state space as an \textit{instantaneous entrance boundary} (see \cite[Definition 1.1]{entranceBdry}) and extend $X$ to a Feller process on $\bar{\bN}$, that is,
there exists a probability measure $\bP_{\infty}$ (or $\bE_{\infty}$ for the expectation) on $\bD(\bar{\bN})$, the set of c\`adl\`ag paths on $\bar{\bN}$, and for every function $f \in C_{b}(\bar{\bN}) := \{ g: \bar{\bN} \to \bR \mid \lim_{x \to \infty} g(x) = g(\infty) \}$ the following holds:
\begin{enumerate}
	\item $P_{t}f(x) := \bE_{x}[f(X_{t}),\tau_{0} > t] \in C_{b}(\bar{\bN})$.
	\item $\lim_{t \to 0+}P_{t}f(x) = f(x) \quad (x \in \bar{\bN})$
\end{enumerate}
(see e.g., Kallenberg \cite[p.369]{Kallenberg}).
Let us consider the following condition (see \cite[Definition 1.1]{entranceBdry}):
\begin{align}
	\text{For every $t > 0$,} \quad \lim_{x \to \infty} \liminf_{y \to \infty} \bP_{y}[\tau_{x} \leq t] = 1. \label{instantaneousEntrance}
\end{align}

We show the entrance condition is equivalent to \eqref{instantaneousEntrance}.

\begin{Prop}
	The condition \eqref{instantaneousEntrance} holds if and only if the boundary $\infty$ is entrance.
\end{Prop}

\begin{proof}
	From \cite[Lemma 1.2]{entranceBdry}, the condition \eqref{instantaneousEntrance} is equivalent to  $\sup_{x \geq 0} \bE_{x}[\tau_{0}] < \infty$. 
	From \eqref{potentialDensityOneside}, it holds
	\begin{align}
		\sup_{x \geq 0}\bE_{x}[\tau_{0}] = \lim_{x \to \infty} \sum_{y \geq 0} (W(0,y) - W(x,y)) = \sum_{y \geq 0} W(0,y). \nonumber
	\end{align}
\end{proof}

As a direct consequence of \cite{entranceBdry}, we obtain the following:

\begin{Prop}[{\cite[Theorem 2.2]{entranceBdry}}]
	Suppose the boundary $\infty$ is entrance.
	Then the process $X$ can be extended to a Feller process on $\bar{\bN}$ and the following holds:
	\begin{align}
		\bP_{\infty}[X_{0} = \infty] \quad \text{and} \quad \bP_{\infty}[X_{t} \in \bN \quad \text{for every $t > 0$}] = 1. \nonumber
	\end{align}
\end{Prop}

\begin{Prop}[{\cite[Corollary 2.7]{entranceBdry}}]
	Suppose the boundary $\infty$ is entrance.
	Then the following convergence on $D(\bar{\bN})$ holds in the sense of Skorokhod's $J_{1}$-topology (see e.g.,\cite[Chapter 3.3]{StochasticProcessLimits}):
	\begin{align}
		\bP_{x}[X \in de] \xrightarrow[]{x \to \infty} \bP_{\infty}[X \in de] \quad (e \in D(\bar{\bN})). \nonumber
	\end{align}
\end{Prop}

We show the Yaglom limit exists for $X$ starting from $\infty$.

\begin{Thm}
	Suppose $\infty$ is entrance.
	Then for every $A \subset \bN$, the following holds:
	\begin{align}
		\lim_{t \to \infty} \mathrm{e}^{\lambda_{0}t}\bP_{\infty}[X_{t} \in A, \tau_{0} > t] = \frac{\nu_{\lambda_{0}}(A)}{\rho \lambda_{0}}. \nonumber
	\end{align}
\end{Thm}

\begin{proof}
	Take $\theta > \lambda_{0}$.
	From \cite[Proposition 2.4]{entranceBdry}, there exists $x_{\theta} > 0$ such that $\bE_{\infty}[\mathrm{e}^{\theta \tau_{x_{\theta}}}] < \infty$.
	Since it holds
	\begin{align}
		&\mathrm{e}^{\lambda_{0}t}|\bP_{\infty}[X_{t} \in A, \tau_{0} > t] - \bP_{\infty}[X_{t} \in A, \tau_{0} > t, \tau_{x_{\theta}} \leq t]| \nonumber \\
		\leq &\mathrm{e}^{\lambda_{0}t} \bP_{\infty}[\tau_{x_{\theta}} > t] \nonumber \\
		\leq &\mathrm{e}^{-(\theta - \lambda_{0})t} \bE_{\infty}[\mathrm{e}^{\theta \tau_{x_{\theta}}}] \xrightarrow[]{t \to \infty} 0, \nonumber
	\end{align}
	it is enough to show
	\begin{align}
		\lim_{t \to \infty}\mathrm{e}^{\lambda_{0}t} \bP_{\infty}[X_{t} \in A, \tau_{0} > t,\tau_{x_{\theta}} \leq t] = \frac{\nu_{\lambda_{0}}(A)}{\rho \lambda_{0}}. \nonumber
	\end{align}
	From the Markov property and Theorem \ref{thm:YaglomLimit} (iii),
	we see $\mathrm{e}^{\lambda_{0}t}\bP_{x_{\theta}}[X_{t} \in A, \tau_{0} > t]$ is bounded in $t \geq 0$ and converges to $ (\rho \lambda_{0})^{-1} Z^{(-\lambda_{0})}(x_{\theta}) \nu_{\lambda_{0}}(A)$ as $t \to \infty$.
	From the dominated convergence theorem we have
	\begin{align}
		&\mathrm{e}^{\lambda_{0}t} \bP_{\infty}[X_{t} \in A, \tau_{0} > t,\tau_{x_{\theta}} \leq t] \nonumber \\
		= &\int_{0}^{t} \mathrm{e}^{\lambda_{0}(t-u)}\bP_{x_{\theta}}[X_{t-u} \in A, \tau_{0} > t-u] \mathrm{e}^{\lambda_{0}u}\bP_{\infty}[\tau_{x_{\theta}} \in du] \nonumber \\
		\xrightarrow[]{t \to \infty} &\frac{Z^{(-\lambda_{0})}(x_{\theta})}{\rho \lambda_{0}}\nu_{ \lambda_{0}}(A)\bE_{\infty}[\mathrm{e}^{\lambda_{0}\tau_{x_{\theta}}}]. \nonumber
	\end{align}
	From \cite[Proposition 2.4]{entranceBdry} and Lemma \ref{lem:lambdaX} (i), we see 
	\begin{align}
		Z^{(-\lambda_{0})}(x_{\theta}) \bE_{\infty}[\mathrm{e}^{\lambda_{0}\tau_{x_{\theta}}}]
		&= Z^{(-\lambda_{0})}(x_{\theta}) \lim_{z \to \infty} \bE_{z}[\mathrm{e}^{\lambda_{0}\tau_{x_{\theta}}}] \nonumber \\
		&= Z^{(-\lambda_{0})}(x_{\theta}) \lim_{z \to \infty} \frac{Z^{(-\lambda_{0})}(z)}{Z^{(-\lambda_{0})}(x_{\theta})} \nonumber \\
		&= \lim_{z \to \infty} Z^{(-\lambda_{0})}(z) \nonumber \\
		&= 1. \nonumber
	\end{align}
\end{proof}

\bibliography{arxiv03.bbl}
\bibliographystyle{plain}

\end{document}